\documentclass{amsart}

\usepackage{amsrefs} 

\usepackage{amsmath}
\usepackage{amssymb}
\usepackage{mathrsfs}  
\usepackage{hyperref}
\usepackage{enumitem}

\newtheorem{theorem}{Theorem}[section]
\newtheorem{lemma}[theorem]{Lemma}

\newtheorem{proposition}[theorem]{Proposition}

\newtheorem{remark}[theorem]{Remark}
\newtheorem{example}[theorem]{Example}

\theoremstyle{definition}
\newtheorem{definition}[theorem]{Definition}

\newcommand\R{\mathbb{R}}
\newcommand\Z{\mathbb{Z}}
\newcommand\T{\mathbb{T}}
\newcommand\C{\mathbb{C}}
\newcommand\N{\mathbb{N}}

\newcommand{\qtq}[1]{\quad\text{#1}\quad}

\newcommand\eps{\varepsilon}

\let\Re=\undefined\DeclareMathOperator{\Re}{Re}

\DeclareMathOperator{\sgn}{sgn}

\newcommand{\del}{\delta}
\newcommand{\mc}[1]{\mathcal{#1}}

\newcommand{\ol}[1]{\overline{#1}}
\newcommand{\wh}[1]{\widehat{#1}}
\newcommand{\wt}[1]{\widetilde{#1}}

\newcommand{\vk}{\varkappa}
\newcommand{\h}{\mathfrak{h}}

\newcommand{\un}{{\bf{1}}}
\newcommand{\dd}{\mathrm{d}}

\usepackage{xcolor}

\numberwithin{equation}{section}

\allowdisplaybreaks

\begin{document}

\title[GWP for the periodic ILW equation]{Global well-posedness for \\ the ILW equation in $H^s(\T)$ for $s>-\frac 12$}

\author[L.~Gassot]{Louise Gassot}
\address{CNRS and Department of Mathematics, University of Rennes, France}
\email{louise.gassot@cnrs.fr}

\author[T.~Laurens]{Thierry Laurens}
\address{Department of Mathematics, University of Wisconsin--Madison, WI, 53706, USA}
\email{laurens@math.wisc.edu}

\begin{abstract}
We prove that the intermediate long wave (ILW) equation is globally well-posed in the Sobolev spaces $H^s(\mathbb{T})$ for $s > -\frac12$.  The previous record for well-posedness was $s\geq 0$, and the system is known to be ill-posed for $s<-\frac12$.  We then demonstrate that the solutions of ILW converge to those of the Benjamin--Ono equation in $H^s(\T)$ in the infinite-depth limit.

Our methods do not rely on the complete integrability of ILW, but rather treat ILW as a perturbation of the Benjamin--Ono equation by a linear term of order zero.  To highlight this, we establish a general well-posedness result for such perturbations, which also applies to the Smith equation for continental-shelf waves.
\end{abstract}

\maketitle

\tableofcontents

\section{Introduction}

The Benjamin--Ono equation
\begin{equation}
\partial_t u = H \partial_x^2 u - 2u\partial_x u
\label{BO}\tag{BO}
\end{equation}
was introduced in~\cites{Benjamin67,Davis67,Ono77} as an asymptotic model for long internal waves in a deep stratified fluid.  Here, $H$ denotes the Hilbert transform, a Fourier multiplier that writes on the torus as
\begin{equation*}
\wh{Hf}(n) = - i \sgn(n) \wh{f}(n),
\quad n\in\Z.
\end{equation*}

The question of well-posedness of \eqref{BO} in the Sobolev spaces $H^s$ has been intensely pursued over the past half century.  Recently, this effort culminated in the works \cites{GerardKappelerTopalov2020,Killip2024} that established well-posedness for all $s>-\frac12$ on both the line and the torus.  In the latter geometry, this result is sharp in view of \cite{GerardKappelerTopalov2020}, which also demonstrates ill-posedness in $H^{-1/2}(\T)$.

\subsection{Zeroth-order perturbations of the Benjamin--Ono equation}

In this paper, we study real-valued solutions of the perturbed Benjamin--Ono equation
\begin{equation}
\partial_t u = H\partial^2_x u - 2u\partial_x u + Au 
\label{ILW 2}
\end{equation}
on the torus, where $A$ is a Fourier multiplier of order zero satisfying the following conditions. 

\begin{definition}[Zeroth-order perturbations]\label{def:A}
We assume that $A$ is an order-zero Fourier multiplier, so that the symbol $a$ of $A$ is of the form
\begin{equation}
\wh{ Af }(n) = a(n) \wh{f}(n) \qtq{with} \|a\|_{\ell^{\infty}}=\sup_{n\in\Z} |a(n)| < \infty .
\label{F}
\end{equation}
We also assume that $Au$ is real-valued when $u$ is real-valued:
\begin{equation}\label{eq:A-real}
a(-n)=\overline{a(n)},
\quad n\in\Z.
\end{equation}
\end{definition}

Assumption~\eqref{eq:A-real} is natural since the solutions of \eqref{BO} and the models presented below are real-valued.  

Our primary example of such a Fourier multiplier is the following.
\begin{example}
The intermediate long wave equation
\begin{equation}
\partial_t u = -\tfrac{1}{\delta} \partial_x u + T_\del \partial_x^2 u - 2u\partial_x u ,
\label{ILW}\tag{ILW}
\end{equation}
was introduced in \cites{Joseph1977,Kubota78} as a model for long internal waves in fluids of finite total depth.  Here, $\del>0$ describes the depth, and
\begin{equation*}
\wh{T_\del f}(n) = - i \coth(\del n) \wh{f}(n) .
\end{equation*}
This can be rewritten in the form~\eqref{ILW 2}, with
\begin{equation}\label{F:ILW}
Af = ( -\del^{-1} \partial_x + T_\del \partial_x^2 - H\partial^2_x ) f .
\end{equation}
After performing the change of unknown
\begin{equation}
v(t,x)
    :=u(t,x+\delta^{-1}t)
\label{boost}
\end{equation}
to remove the $\del^{-1} \partial_x$ term above, this operator $A$ satisfies~\eqref{F}.  (See \eqref{smoothing 2} for details.)
\end{example}

It is well-known that \eqref{ILW} is completely integrable \cites{Kodama1981,Kodama1982}.  Our methods will not employ this special structure however, but rather only the properties listed in Definition~\ref{def:A}, and so they apply to a much broader class of systems.  In particular, this includes the variant ILW-like model introduced in~\cite{Kubota78} that is not known to be integrable, but has a similar operator $A$ that also satisfies Definition~\ref{def:A}.

In fact, the operator $A$ in \eqref{F:ILW} satisfies a much stronger property than what is required by~\eqref{F}: it is smoothing in an $H^s$ sense, as one may readily verify in Fourier variables (see \eqref{smoothing} for details).  This property was first introduced in \cite{IfrimSaut25}*{Lem.~2.2}, and has since been leveraged in the works~\cites{Chapouto2023,Chapouto2024}.  By comparison, our methods will only rely on~\eqref{F}, which merely requires that the operator $A$ is bounded on $H^s$.  Consequently, this allows us to consider more general dispersion relations, such as the following.

\begin{example}
The Smith equation
\begin{equation}\tag{Smith}
   \partial_t u =  \partial_x(1+|D|^2)^{\frac{1}{2}}u - 2u\partial_x u ,
\label{Smith}
\end{equation}
where
\begin{equation*}
    \wh{|D|u}(n)=|n|\wh{u}(n),
\end{equation*}
was introduced in \cite{Smith1972} as a model for continental-shelf waves.
This is an example of an equation of the form~\eqref{ILW 2} with
\begin{equation*}
\widehat{Af}(n)
    =(in \sqrt{1+n^2}-i\sgn(n)n^2)\widehat{f}(n).
\end{equation*}
One can check using Taylor expansion of $\sqrt{1+n^{-2}}$ as $|n|\to +\infty$ that there holds the uniform estimate in $n$,
\begin{equation*}
\left|in \sqrt{1+n^2}-i\sgn(n)n^2\right|
    =\left|n( \sqrt{1+n^2}-|n|)\right|
    \lesssim 1.
\end{equation*}
This equation is not known to be integrable.
\end{example}

Finally, our class of perturbations also includes the Benjamin--Ono equation with Rayleigh dissipation, in which $Au= - u$, but also anti-dissipation terms such as $Au= u$. However, it does not work for more various dissipation terms such as the Reynolds dissipation $Au=-\partial_{xx}u$ (see~\cite{GrimshawSmythStepanyants2018} and references therein for an introduction and numerical study).  Our methods are not restricted to Hamiltonian perturbations of~\eqref{BO}, and could also be extended to the damped Benjamin--Ono equation studied in~\cite{Gassot21-damped} for which the perturbation is $-\langle u,\cos\rangle \cos -\langle u,\sin\rangle \sin$.

\subsection{Main results}

For $s\in\R$, we denote by $H^s(\T)=H^s(\T;\R)$ the space of real-valued functions of class $H^s$.

Our first main result is the well-posedness of equation~\eqref{ILW 2} in $H^s(\T)$ when $s>-\frac 12$.

\begin{theorem}[Global well-posedness]\label{t:wp 2}
Fix $s>-\frac12$.  If the Fourier multiplier $A$ satisfies Definition~\ref{def:A}, then equation~\eqref{ILW 2} is locally well-posed for initial data in $H^s(\T)$.
Moreover, if $-\frac 12<s<0$, equation~\eqref{ILW 2} is globally well-posed  for initial data in $H^s(\T)$.
\end{theorem}

In the case $A\equiv 0$, this result is sharp in the scope of $H^s(\T)$ spaces in view of the ill-posedness result in~\cite{GerardKappelerTopalov2020}.

\begin{remark}
Using the change of unknowns~\eqref{boost}, we deduce in particular that equation~\eqref{ILW} is globally well-posed for initial data in $H^s(\T)$, $-\frac 12<s<0$.  As we will discuss further below, when combined with prior work this yields global well-posedness for all $s>-\frac12$.  On the other hand, the authors of \cite{Chapouto2024} demonstrated that the equation is ill-posed for $s < -\frac12$.
\end{remark}

\begin{remark}
For~\eqref{ILW 2}, local well-posedness can also be extended to global well-posedness for $s\geq 0$ using the polynomial conserved quantities of~\eqref{BO} and a Gronwall argument.  We will focus on the case $-\frac12 < s < 0$, which for \eqref{ILW} covers all cases not previously known.
\end{remark}

Our second main result is the convergence of solutions of~\eqref{ILW} to a solution of~\eqref{BO} in $H^s(\T)$ as $\delta\to+\infty$.
\begin{theorem}[Convergence of ~\eqref{ILW} to~\eqref{BO}]\label{t:conv}
Fix $-\frac12 < s < 0$.  Given initial data $u(0) \in H^s(\T)$, the global solutions $u^\del(t)$ of \eqref{ILW} converge as $\del\to\infty$ to the global solution $u^\infty(t)$ of \eqref{BO} in $\mc{C}([-T,T];H^s(\T))$ for any $T>0$.
\end{theorem}

As in the case of our well-posedness result for \eqref{ILW}, this provides the last remaining piece in order to account for the full range $s>-\frac12$ of the sharp well-posedness theory for \eqref{BO}.

\subsection{Strategy of proof}

The first ingredient of our proof is the Birkhoff map: a nonlinear Fourier transform that sends the Benjamin--Ono dynamics into an infinite system of linear ODEs.  We introduce this transformation in Section~\ref{s:Birkhoff}.
This Birkhoff map was constructed in~\cite{GerardKappeler2021} at regularity $L^2_0(\T)$, then extended to $H^s_0(\T)$ for $-\frac 12<s<0$ in~\cite{GerardKappelerTopalov2020}.  Here, $H^s_0(\T)$ denotes the subspace of $H^s(\T)$ consisting of functions with zero mean.  In general, the mean value of solutions is not preserved by the flow of \eqref{ILW 2}, but we explain in Section~\ref{s:Birkhoff} how to enforce this condition through the change of unknown~\eqref{boost 2}.  Thanks to analyticity properties of the Birkhoff map~\cite{GerardKappelerTopalov2021analyticity}, we are then able to write a Duhamel formula \eqref{duhamel} in Birkhoff coordinates that expresses solutions to \eqref{ILW 2} as perturbations of the \eqref{BO} flow.

In Section~\ref{s:s-pos}, we prove that equation \eqref{ILW 2} is locally well-posed in $H^s_0(\T)$ for $s\geq 0$ by implementing a fixed-point argument based on this Duhamel formulation.  It turns out that in the Birkhoff coordinates, the perturbation is bounded and Lipschitz on finite balls corresponding to $L^2_0(\T)$ regularity, allowing for a fixed-point argument.  A similar strategy was successfully implemented in~\cite{Gassot21-damped} for a Benjamin--Ono equation with a damping term on the first Fourier modes.   In~\cite{BambusiGerard24}, a Nekhoroshev-type theorem was proved for a large class of small Hamiltonian perturbations of the Benjamin--Ono equation, by using crucially the analyticity of the Birkhoff map in the energy space $H^{1/2}_0(\T)$.

However, such strategies become much more difficult to implement in $H^s_0(\T)$ for $s<0$, due to the lack of uniform continuity of the flow map of the Benjamin--Ono equation as such low regularities~\cite{GerardKappelerTopalov2021}. Note that even though it was shown that the flow map of~\eqref{BO} is never uniformly continuous on bounded subsets of $H^s(\R)$ when $s>0$~\cite{KochTzvetkov05}, the solution map is actually analytic on $H^s_0(\T)$ when $s\geq 0$~\cite{GerardKappelerTopalov2021analyticity}.  However, uniform continuity of the \eqref{BO} flow on bounded subsets of $H^s_0(\T)$ fails for $s<0$~\cite{GerardKappelerTopalov2021}.

In order to go beyond $s\geq 0$, we introduce our second main ingredient in Section~\ref{s:a-priori}: we show that for $-\frac12<s<0$, the orbits of~\eqref{ILW 2} on bounded time intervals are a-priori relatively compact in $H^s(\T)$.  By the Arzel\`a--Ascoli theorem, boundedness and equicontinuity (see \eqref{equicty 2} for details) pose necessary and sufficient conditions for a set to be relatively compact in $H^s(\T)$.  We verify that orbits satisfy both of these criteria using a one-parameter family of quantities $\beta(\kappa;u)$ defined in \eqref{beta} that are conserved by the \eqref{BO} flow.  This same quantity was employed in \cite{Killip2024} as the starting point for the method of commuting flows in order to prove Theorem~\ref{t:wp 2} in the case $A\equiv 0$ on both the line and the torus.

In the general case $A\not\equiv 0$, we prove that these quantities can grow at most exponentially in time.  For \eqref{ILW}, this result was recently established in \cite{Chapouto2024}.  However, their argument relies on smoothing properties for the operator $A + \del^{-1}\partial_x$, and so does not apply to other models such as the \eqref{Smith} equation.

In Section~\ref{s:WP}, we prove that equation \eqref{ILW 2} is globally well-posed in $H^s_0(\T)$ for $-\frac12<s<0$.  The central question we need to address is: Does any sequence of smooth initial data $u_j(0)$ that converges in $H^s(\T)$ lead to a sequence of solutions $u_j(t)$ that converges in $H^s(\T)$ uniformly on bounded time intervals $I$?  Courtesy of our a-priori results, we now know that the Birkhoff coordinates present a bi-Lipschitz change of variables on the set $\{ u_j(t) : j\geq 1,\ t\in I\}$.  This provides the starting point for our fixed-point argument.

However, the simple fixed-point argument from Section~\ref{s:s-pos} in the case $s\geq 0$ breaks down for $s<0$.  Instead, in order to close our argument, we need to estimate the difference of solutions in a space that is one degree of regularity lower than that of the initial data (see Lemma~\ref{lem:low} for details).  Nevertheless, we are then able to upgrade this convergence to the necessary regularity by a second application of our a-priori equicontinuity result, which corresponds to tightness of the sequence of Birkhoff coordinates.

Finally, in Section~\ref{s:ILW} we focus our attention on \eqref{ILW} and prove Theorem~\ref{t:conv} regarding the infinite-depth limit of solutions.  Our starting point is again the Duhamel formula~\eqref{duhamel} in Birkhoff coordinates, which expresses \eqref{ILW} as a $\del$-dependent perturbation of the \eqref{BO} flow.  In comparing the difference of solutions to \eqref{ILW} and \eqref{BO}, we encounter a loss-of-derivatives phenomenon analogous to the one in Section~\ref{s:WP}.  This issue is also resolved by equicontinuity, because our a-priori estimate for $\beta(\kappa;u)$ holds uniformly in the limit $\del\to\infty$.

\subsection{Perturbations of the Benjamin--Ono equation}
Let us now review some known results about the dynamics of perturbed Benjamin--Ono equations.

In~\cite{Saut79} and~\cite{AbdelouhabBFS89}, the authors show that equations~\eqref{ILW} and~\eqref{Smith} are globally well-posed in $H^s(\R)$ for $s>\frac32$; moreover, solutions to~\eqref{ILW} converge to those of~\eqref{BO} as the depth parameter goes to infinity. These results also extend to the periodic case. It is also known that the flow map in $H^s(\R)$ cannot be of class $\mathcal{C}^2$ for these equations~\cite{MolinetSautTzvetkov01}.  For both \eqref{ILW} and \eqref{Smith}, the global well-posedness threshold was then lowered to the energy space $s=\frac12$~\cite{MolinetVento15}, the result being valid on the line and on the torus. Excluding the endpoint case, convergence for \eqref{ILW} in the infinite-depth limit was established in \cite{Li2024} for $s>\frac12$.  Recently, in~\cite{IfrimSaut25}, global well-posedness for equation~\eqref{ILW} was proved in $L^2(\R)$ along with dispersive estimates.
In~\cite{Chapouto2023}, global well-posedness for equation~\eqref{ILW}  was proved in $L^2(\T)$, as well as the convergence of the~\eqref{ILW} dynamics to the~\eqref{BO} dynamics, by treating~\eqref{ILW} as a perturbation of~\eqref{BO}. The equation is also known to be ill-posed when $s<-\frac 12$ on both the line and the torus~\cite{Chapouto2024}. Concerning the long-time dynamics, there is numerical evidence that the soliton resolution conjecture holds for equations~\eqref{Smith} and~\eqref{ILW}; see~\cite{KleinSautBook} and references therein.

In the periodic Benjamin--Ono equation with a damping term on its first Fourier modes, it was possible to show that the Sobolev norms of the trajectories associated to initial data in $L^2(\T)$ stay bounded globally in time in the higher-order Sobolev spaces $H^s(\T)$ for $0\leq s<3/2$~\cite{Gassot21-damped}. In contrast, such a damping term was initially introduced in~\cite{GerardGrellier19} for the integrable Szeg\H{o} equation, where the authors could prove growth of Sobolev norms in infinite time.

Let us also mention how extensions of the KAM theory have been implemented to study  small perturbations of integrable equations.
In~\cite{LiuYuan2011} and~\cite{MiZhang2014}, the authors show the existence of quasi-periodic trajectories for unbounded Hamiltonian perturbations of~\eqref{BO}. For a cubic Benjamin--Ono equation, it was shown in~\cite{Baldi2013} that there exist periodic solutions to some non-Hamiltonian reversible perturbations.  Approximations of trajectories on large time scales were derived in~\cite{BernierGrebert2020}  for Benjamin--Ono equations with general nonlinearities with generic small initial data. 
In this direction, the Nekhoroshev theorem presented in~\cite{BambusiGerard24} also provides information on the long time dynamics by relying on the Birkhoff coordinates. The authors show that the actions of the Benjamin--Ono equations (given by the square modulus of the Birkhoff coordinates) remain close to their initial values for times wich are exponentially long with a power $\eps^{-1}$ if the perturbation is of size~$\eps$. 

Concerning equation~\eqref{ILW 2}, we only need to derive an exponential a-priori bound on the $H^s$ norm of solutions in order to prove our global well-posedness result. The growth of Sobolev norms for the trajectories of~\eqref{ILW 2}, or more generally, the qualitative long-time behavior of solutions, therefore remains an interesting open question. 

In the context of the Korteweg--de Vries (KdV) equation, the second author established an exponential bound in \cite{Laurens2023} as a stepping stone towards global well-posedness for $H^{-1}(\R)$ perturbations $u$ of a bounded solution $V$ of KdV.  This matches the regularity for the sharp well-posedness theory for the KdV equation.  Although the combined waveform $u+V$ is a solution to KdV, the discrepancy between the boundary conditions of $V$ as $x\to\pm\infty$ can drive the growth of the $H^s$ norm of solutions.  

Recently, the work \cite{Ifrim2025} demonstrated a-priori estimates in $H^{-1}(\R)$ for a family of non-integrable perturbations of the KdV equation.  These estimates are local in time, but they allow for a much broader class of perturbations, including linear terms up to second order as well as quadratic terms in the PDE.

\subsection{Acknowledgments}

The authors would like to thank P. Gérard and D. Pilod for interesting discussions.

L. Gassot was supported by the France 2030 framework program, the Centre Henri Lebesgue ANR-11-LABX-0020-01, the ANR project HEAD--ANR-24-CE40-3260, and the AIS Rennes Métropole.
 T.~Laurens was supported by NSF CAREER grant DMS-1845037.

\section{Birkhoff coordinates}\label{s:Birkhoff}

\subsection{Properties of the Birkhoff map}

Our proof relies on the existence of Birkhoff coordinates for equation~\eqref{BO}. For $s>-\frac 12$, the Birkhoff map is denoted
\begin{equation*}
\Phi:u\in H^s_0(\T)\mapsto (\zeta_n(u))_{n\geq 1}\in \h^{s+\frac 12},
\end{equation*}
where $H^s_0(\T)$ is the subspace of functions of $H^s(\T)$ with zero mean, and
\begin{equation*}
\h^{ s+\frac 12}
    = \bigg\{ (z_n)_{n\geq 1}\in \C^\N : \|z_n\|_{\h^{s+\frac12}_n}^2 = \sum_{n\geq 1} n^{2s+1}|z_n|^2<+\infty \bigg\}.
\end{equation*}
As we will discuss below, this mapping is continuously differentiable.  We denote its functional derivative by
\begin{equation*}
\dd_u \Phi[f] = \frac{\dd}{\dd s} \Phi( u + sf) \bigg|_{s=0} .
\end{equation*}

According to~\cites{GerardKappeler2021}, we know that if $u$ is an $L^2$ solution to \eqref{BO}, then for every $n\geq 1$,
\begin{equation}
\frac{\dd}{\dd t} \zeta_n(u(t)) = i\omega_n(u(t))\zeta_n(u(t)) ,
\label{zeta dot 3}
\end{equation}
where
\begin{equation}\label{eq:omega_n}
\omega_n(u) = n^2 - 2\sum_{k\geq 0} \min\{ k,n\} |\zeta_k(u)|^2.
\end{equation}
Moreover, the flow defined by~\eqref{BO} for smooth initial data extends continuously to $H^s(\T)$ when $s>-\frac 12$, and this extension is a solution to~\eqref{zeta dot 3} in Birkhoff coordinates~\cite{GerardKappelerTopalov2020}.  

We will use the following properties of the Birkhoff map.
\begin{proposition}[\cite{GerardKappelerTopalov2020}*{Theorem 6} and~\cite{GerardKappelerTopalov2021analyticity}]\label{t:birkhoff-bdd}
For any $s > -\frac12$, the Birkhoff map sends bounded subsets of $H^s_0(\T)$ to bounded subsets of $\h^{s+\frac 12}$, and similarly for the inverse map.

Moreover, the Birkhoff map is a real-analytic diffeomorphism. In particular, for every $u\in H^s_0(\T)$, there is $C>0$ and a neighborhood $V$ of $u$ in $H^s_0(\T)$ such that for every $\wt{u}\in V$ and $f\in H^s(\T)$,
\begin{align}
\| \dd_u \Phi [f] \|_{\h^{s+\frac12}} &\leq C \| f \|_{H^s} ,
\label{dzeta} \\
\| \dd_u\Phi[f] - \dd_{\wt{u}}\Phi[f]  \|_{\h^{s+\frac12}} &\leq C\| u - \wt{u} \|_{H^s} \| f \|_{H^s}.
\label{dzeta 2}
\end{align}
Similar estimates hold for the inverse map $\Phi^{-1}$.
\end{proposition}
In particular, \eqref{dzeta} and \eqref{dzeta 2} hold uniformly on compact sets of $u,\wt{u}\in H^s_0(\T)$.

\subsection{Zero-mean condition}

As the Birkhoff map is defined on the subspace $H^s_0(\T)$ of elements of $H^s(\T)$ with zero mean, we will now explain how to enforce this condition for solutions to equation~\eqref{ILW 2}. 

 Let us first assume that $a(0)=0$.   In this case, the mean value $\frac{1}{2\pi}\int_0^{2\pi} u\,\dd x$ is a conserved quantity for the equation~\eqref{ILW 2}, and so it is always possible to enforce the condition that the solution takes its values in $H^s_0(\T)$ thanks to the Galilei symmetry
\[
\widetilde{u}(t,x)=u(t,x-2ct)+c.
\]

Let us now assume that $a(0)\neq 0$. Consider the change of unknowns
\begin{equation}
v(t,x)=u(t,x-2d(t))+c(t),
\label{boost 2}
\end{equation}
where we choose $c(t)$ and $d(t)$ so that $v(0)=u(0)$, and so that $v$ has mean zero. Then $v$ solves the following equation:
\begin{align*}
\partial_t v
	&= \partial_x(H\partial_x (v-c)-(v-c)^2)+Av-Ac(t)-2d'(t)\partial_x v+c'(t)\\
	&=\partial_x(H\partial_x v-v^2)+Av +2(c(t)-d'(t))\partial_x v -a(0)c(t)+c'(t).
\end{align*}
We choose $c$ and $d$ so that 
\[
-a(0)c(t)+c'(t)=0
\quad\text{and}\quad d'(t)=c(t),
\]
with the initial condition $c(0)=-\frac{1}{2\pi}\int_0^{2\pi} u$ and $d(0)=0$.  We find $c(t)=c(0)e^{a(0)t}$, and $d(t)=\frac{c(0)}{a(0)}(e^{a(0)t}-1)$.

Then, the unknown $v$ is a solution to~\eqref{ILW 2} with zero mean. It is therefore enough to check that well-posedness holds on $H^s_0(\T)$ for equation~\eqref{ILW 2}.

\subsection{Duhamel formula}

Given a smooth solution $u$ to \eqref{ILW 2} in $H^s_0(\T)$ for some $s>-\frac 12$, we will write a Duhamel formula adapted to the Birkhoff coordinates.  By the chain rule, given that the time-evolution for the Benjamin--Ono part of the equation is given by \eqref{zeta dot 3}, we see that 
\begin{equation}
\frac{\dd}{\dd t} \zeta_n(u(t)) = \dd_{u(t)}\zeta_n\,[\partial_t u(t)] =  i\omega_n(u(t)) \zeta_n(u(t)) + \dd_{u(t)}\zeta_n\,[ Au(t)] .
\label{zeta dot}
\end{equation}
Setting
\begin{equation*}
\Omega_n(t;u) = \int_0^t \omega_n(u(\tau))\,\dd\tau ,
\end{equation*}
it follows that
\begin{equation*}
\frac{\dd}{\dd t} ( e^{-i\Omega_n}  \zeta_n ) = e^{-i\Omega_n} \dd_u\zeta_n [Au] ,
\end{equation*}
and so
\begin{equation}
\zeta_n(u(t)) = e^{i\Omega_n(t;u)} \zeta_n(u(0)) +  \int_0^t e^{i[\Omega_n(t;u) - \Omega_n(\tau;u)]} \dd_{u(\tau)}\zeta_n [Au(\tau)]\, \dd\tau .
\label{duhamel}
\end{equation}

\section{Local well-posedness for $s\geq 0$}\label{s:s-pos}

Let us first provide a simple proof of local well-posedness in $H^s_0(\T)$ when $s\geq 0$ using a fixed-point argument.

Let $u_0\in H^s_0(\T)$. We denote by $u^{\infty}(t)$ the Benjamin--Ono evolution associated to the initial data $u_0$: for every $n\geq 1$, and $t\in\R$,
\begin{equation*}
\zeta_n(u^{\infty}(t))=e^{it\omega_n(u_0)}\zeta_n(u_0).
\end{equation*}

 We show that if $T$ is small enough, then there is a constant $R>0$ such that the Duhamel map
\begin{align*}
u
    &\mapsto \Phi^{-1} \left(e^{i\Omega_n(t;u)} \zeta_n(u_0) +  \int_0^t e^{i[\Omega_n(t;u) - \Omega_n(\tau;u)]} \dd_{u(\tau)}\zeta_n [Au(\tau)]\, \dd\tau\right) 
\end{align*}
sends the ball $B_{T,R}=\{u\in \mc{C}([-T,T]; H^s_0(\T)) : \|u-u^{\infty}\|_{\mc{C}_TH^s}< R\}$ to itself and is a contraction on this ball. We use the short notation
\[
\|u\|_{\mc{C}_TH^s}:=\|u\|_{\mc{C}([-T,T];H^s(\T))} = \sup_{t\in [-T,T]} \|u(t)\|_{H^s} .
\]

Since the set of trajectories $\{u^{\infty}(t) : t\in[-T,T]\}$ is a compact set in $H^s_0(\T)$,  we assume up to taking $R$ small enough that the bounds in Proposition~\ref{t:birkhoff-bdd}, in particular~\eqref{dzeta} and~\eqref{dzeta 2}, hold uniformly on this ball.

Then we estimate
\begin{multline*}
\|\zeta_n(u(t))-\zeta_n(u^\infty(t))\|_{\mc{C}_T\h^{s+\frac12}}\\
    \leq \|(e^{i\Omega_n(t;u)}-e^{it\omega_n(u_0)})\zeta_n(u_0)\|_{\mc{C}_T\h^{s+\frac12}}
    +T \|\dd_{u(\tau)}\zeta_n[Au(\tau)]\|_{\mc{C}_T \h^{s+\frac 12}}.
\end{multline*}
Using $t\omega_n(u_0)=\Omega_n(t;u^{\infty})$, \eqref{eq:omega_n}, and \eqref{dzeta}, we note that
\begin{align}
|\Omega_n(t;u)-t\omega_n(u_0)|
    &\leq T \bigg\| \sum_{k\geq 1} k \Big| |\zeta_k(u(t))|^2-|\zeta_k(u^{\infty}(t))|^2 \Big| \bigg\|_{\mc{C}_T}
    \\
    &\leq CT\|\Phi(u) -\Phi(u^{\infty})\|_{\mc{C}_T\h^{\frac 12}}\\
    &\leq CT\|u -u^{\infty}\|_{\mc{C}_TH^s}.
    \label{eq:omega-diff}
\end{align}
Moreover, by \eqref{dzeta} and \eqref{F} we have
\begin{equation*}
\|\dd_{u(\tau)}\zeta_n[Au(\tau)]\|_{\mc{C}_T \h^{s+\frac 12}}
    \leq C\|Au\|_{\mc{C}_TH^s}
    \leq C\|u\|_{\mc{C}_TH^s}.
\end{equation*}
Combining the previous two estimates with uniform bounds on the differential of the inverse map $\Phi^{-1}$, we find that for $T$ small enough,
\begin{align*}
\|u-u^{\infty}\|_{\mc{C}_TH^s}
    &\leq C\|\Phi(u)-\Phi(u^{\infty})\|_{\mc{C}_T\h^{s+\frac 12}}\\
    &\leq C'T \|u -u^{\infty}\|_{\mc{C}_TH^s}
    +C'T \|u\|_{\mc{C}_TH^s}.
\end{align*}
We conclude that the Duhamel map sends the ball $B_{T,R}$ to itself if $T$ is chosen small enough. A similar argument shows that the Duhamel map is a contraction on this ball. 

Note that estimate~\eqref{eq:omega-diff} was made possible thanks to the fact that $\|\Phi(u) -\Phi(u^{\infty})\|_{\mc{C}_T\h^{\frac 12}}$ is controlled by $\|\Phi(u) -\Phi(u^{\infty})\|_{\mc{C}_T\h^{s+\frac 12}}$ when $s\geq 0$. However, this fails when $s<0$.

\section{Relative compactness of trajectories}\label{s:a-priori}

In this section, we show the following a-priori properties on the trajectories of~\eqref{ILW 2}.

\begin{theorem}\label{t:ap}
If $\mc{F}\subset H^\infty$ is a set of smooth initial data that is relatively compact in $H^s(\T)$, then for any $T>0$ the set of orbits under \eqref{ILW 2}:
\begin{equation}
\mc{F}^*_T = \{ u(t) : u(0)\in \mc{F},\ t\in[-T,T] \}
\label{F*}
\end{equation}
is also relatively compact in $H^s(\T)$.
\end{theorem}

We also show that this property is uniform with respect to the depth parameter $\delta$ for equation~\eqref{ILW}.
\begin{theorem}\label{t:ap ILW}
If $\mc{F}\subset H^\infty(\T)$ is a set of smooth initial data that is relatively compact in $H^s(\T)$, then for any $T>0$ the set of all orbits under \eqref{ILW} for $\del\geq 1$:
\begin{equation*}
\mc{F}^*_T = \{ u^\del(t) : u(0)\in \mc{F},\ t\in[-T,T],\ \del\geq 1 \}
\end{equation*}
is also relatively compact in $H^s(\T)$.
\end{theorem}

\subsection{Sobolev norms}

We denote the Fourier coefficients as follows:
\begin{equation*}
\wh{f}(n) = \tfrac{1}{2\pi} \int_0^{2\pi} e^{-inx} f(x)\,\dd x  ,
\end{equation*}
so that
\begin{equation*}
f(x) = \sum_{n \in \Z} \wh{f}(n) e^{inx} \qtq{and}
\|f\|_{L^2}^2 = \sum_{n \in \Z} |\wh{f}(n)|^2.
\end{equation*}

On $H^r(\T)$, we define the norms
\begin{equation*}
\|f\|_{H^r_\kappa}^2 = \sum_{n \in \Z} |\wh{f}(n)|^2 ( |n|+\kappa)^{2r} .
\end{equation*}
For $\kappa > 0$ fixed, this is equivalent to the usual norm on $H^r(\T)$.  However, the behavior of these norms in the limit $\kappa\to\infty$ will play a key role in our analysis, as we will explain in the next subsection.

For the remainder of the paper, we will fix
\begin{equation}
s \in (-\tfrac12,0), \qtq{and define} \eps = \tfrac12 ( \tfrac12 - |s|) \in (0,\tfrac14) .
\label{s}
\end{equation}
As $s+1 > \frac12$, then $H^{s+1}(\T)$ embeds into $L^\infty$.  Specifically, using Cauchy--Schwarz in Fourier variables we find that
\begin{equation*}
\| f \|_{L^\infty} \leq \sum_{n\in\Z} |\wh{f}(n)| \lesssim_s \kappa^{-2\eps} \| f \|_{H^{s+1}_\kappa}
\end{equation*}
uniformly for $\kappa\geq 1$.  As a consequence, we see that $H^{s+1}(\T)$ is an algebra and
\begin{equation*}
\| fg \|_{H^{s+1}_\kappa} \lesssim_s \kappa^{-2\eps} \| f \|_{H^{s+1}_\kappa} \| g \|_{H^{s+1}_\kappa} .
\end{equation*}
On the other hand, in \cite{Killip2024}*{Lem.~2.1} the authors show that multiplication by $g\in H^{s+1}(\T)$ is a bounded operator on both $H^{\pm s}(\T)$, and 
\begin{equation*}
\| fg \|_{H^{\pm s}_\kappa} \lesssim_s \kappa^{-2\eps} \| f \|_{H^{\pm s}_\kappa} \| g \|_{H^{s+1}_\kappa} .
\end{equation*}
Combining the previous two inequalities with a complex interpolation argument between $H^{-s}(\T)$ and $H^{s+1}(\T)$, we see that multiplication by $g\in H^{s+1}(\T)$ is also bounded operator on $H^{1/2}(\T)$, with
\begin{equation}
\| fg \|_{H^{1/2}_\kappa} \lesssim_s \kappa^{-2\eps} \| f \|_{H^{1/2}_\kappa} \| g \|_{H^{s+1}_\kappa} \quad\text{uniformly for }\kappa\geq 1.
\label{H1/2}
\end{equation}

\subsection{Relative compactness  and a-priori bounds}

On the torus, the Arzel\`a--Ascoli theorem says that a set $\mc{F}$ in $H^r(\T)$ is relatively compact if and only if it is bounded and equicontinuous in $H^r(\T)$.
We recall that a bounded set $\mc{F}\subset H^r(\T)$  is \emph{equicontinuous} in $H^r(\T)$ if
\begin{equation}
\sup_{u\in\mc{F}}\ \sup_{|y| < \del}\, \| u(\cdot+y) - u(\cdot) \|_{H^r} \to 0 \quad\text{as }\del\to 0.
\label{equicty 2}
\end{equation}

By Plancherel, a bounded set $\mc{F}\subset H^r$ is equicontinuous if and only if it is tight in the Fourier variable:
\begin{equation}
\sup_{u\in\mc{F}}\, \sum_{|n|\geq\kappa} |\wh{f}(n)|^2 ( |n|+1)^{2r} \to 0 \quad\text{as }\kappa\to\infty .
\label{equicty 3}
\end{equation}
From this perspective, it is straightforward to verify the following equicontinuity criterion.
\begin{lemma}[\cite{Killip2024}*{Lem.~2.4}]\label{lem:equicty}
Let $s\in(-\frac 12,0)$ and let $\mc{F}$ be a bounded subset of $H^s(\T)$.  Then $\mc{F}$ is equicontinuous in $H^s(\T)$ if and only if
\begin{equation}
\sup_{u\in\mc{F}}\,\| u\|_{H^s_\kappa} \to 0 \quad\text{as } \kappa\to\infty .
\label{equicty}
\end{equation}
\end{lemma}

It is therefore enough to prove a-priori estimates in terms of $\| u\|_{H^s_\kappa}$ that hold uniformly in the limit $\kappa\to\infty$.
\subsection{Conserved quantities for the Benjamin--Ono equation}

Our a-priori estimates are derived from conserved quantities $\beta_s(\kappa;u)$ for the flow of~\eqref{BO}, that are adapted to the $H^s$ regularity. We will show that these conserved quantities grow in time at most exponentially under the flow of~\eqref{ILW 2}.

The conserved quantities are built from the Lax operator associated to  the Benjamin--Ono equation, which  is defined as follows. For $u\in L^2(\T)$, we denote by
 $T_u$ the (unbounded) Toeplitz operator with symbol $u$:
\begin{equation*}
    T_u:f\in L^2_+(\T)\mapsto \Pi (uf)\in L^2_+(\T),
\end{equation*}
where $L^2_+(\T)=\{f\in L^2(\T) : \forall n<0,\, \wh{f}(n)=0 \}$ is the Hardy space, and  $\Pi$ is the Szeg\H{o} projector onto $L^2_+(\T)$:
\begin{equation*}
    \wh{\Pi(u)}(n)=\un_{n\geq 0} \wh{u}(n).
\end{equation*}
For $u\in L^2(\T)$, it is straightforward to show that the Lax operator  
\begin{equation*}
\mc{L}_u = - i\partial_x - T_u 
\end{equation*}
(acting on $L^2_+(\T)$) is self-adjoint with domain $H^1(\T)\cap L^2_+(\T)$.
However, the definition of Lax operator can also be extended to $u\in H^s(\T)$ when $s>-\frac 12$ (see also~\cite{GerardKappelerTopalov2020}*{Corollary 2}).

\begin{proposition}[\cite{Killip2024}*{Prop.~3.2}]
Let $s,\eps$ be as in \eqref{s}.  Given $u\in H^s(\T)$, there exists a unique self-adjoint, semi-bounded operator $\mc{L}_u$ associated to the quadratic form
\begin{equation*}
f \mapsto \langle f, \mc{L}_0 f \rangle - \int u(x) |f(x)|^2\,\dd x
\end{equation*}
with form domain $H^{1/2}_+$.  Moreover, there exists a constant $C_s \geq 1$ so that whenever
\begin{equation}
\kappa \geq C_s \big( 1 + \|u\|_{H^s_\kappa} \big)^{\frac{1}{2\eps}} ,
\label{k0}
\end{equation}
the resolvent $
(\mc{L}_u+\kappa)^{-1}$ exists and satisfies
\begin{equation*}
\| (\mc{L}_u+\kappa)^{-1} f \|_{H^{1/2}_\kappa} \lesssim \| f \|_{H^{-1/2}_\kappa} \qtq{and}
\| (\mc{L}_u+\kappa)^{-1} f \|_{H^{s+1}_\kappa} \lesssim_s \|f\|_{H^{s}_\kappa} .
\end{equation*}
\end{proposition}

Using the Lax operator, it is possible to construct conservation laws for the flow of~\eqref{BO} that have the same level of regularity as the Sobolev spaces $H^s(\T)$. Denoting 
\begin{equation*}
\beta(\kappa;u):=\langle \Pi u,(\mc{L}_u+\kappa)^{-1}\Pi u\rangle,
\end{equation*}
we will see that the relevant conservation law at $H^s$ regularity is
\begin{equation*}
\beta_s(\kappa;u) := \int_{\kappa}^\infty \beta(\vk;u) \vk^{2s}\,\dd\vk.
\end{equation*}

\begin{proposition}[\cite{Killip2024}*{Prop.~4.1 and 4.3}]
There exists a constant $C_s\geq 1$ so that for any $u\in H^s$ and $\kappa$ satisfying \eqref{k0}: 
\begin{enumerate}
\item The function $m(x;\kappa,u) = (\mc{L}_u+\kappa)^{-1} \Pi u$ is in $H^{s+1}_+(\T)=H^{s+1}(\T)\cap L^2_+(\T)$, and satisfies
\begin{equation}
\| m \|_{H^{1/2}_\kappa} \lesssim \| u \|_{H^{-1/2}_\kappa} \qtq{and}
\| m \|_{H^{s+1}_\kappa} \lesssim_s \| u \|_{H^s_\kappa} .
\label{m}
\end{equation}
\item The quantity $\beta(\kappa;u)$ is a finite, real-valued, and real-analytic function of $u$ that satisfies
\begin{equation}
\beta(\kappa;u) 
    = \int_0^{2\pi} u(x) m(x;\kappa,u)\, \dd x
    = \int_0^{2\pi} u(x) \ol{m}(x;\kappa,u)\, \dd x
    \label{beta}
\end{equation}
and
\begin{equation}
\dd_u\beta[f] = \int_0^{2\pi} \big( |m|^2 + m + \ol{m} \big)(x;\kappa,u) f(x)\, \dd x .
\label{dbeta}
\end{equation}
\item The \eqref{BO} flow conserves $\beta$: for all $u\in H^\infty$, we have
\begin{equation}
\dd_u\beta [ H\partial^2_x u - 2u\partial_xu ] = 0 .
\label{beta dot}
\end{equation}
\item The quantity $\beta_s(\kappa;u)$
satisfies
\begin{equation}
C_s^{-1} \| u\|_{H^s_\kappa}^2 \leq \beta_s(\kappa;u) \leq C_s \| u \|_{H^s_\kappa}^2 .
\label{beta s}
\end{equation}
\end{enumerate}
\end{proposition}

\subsection{A-priori estimates}

We now show the following a-priori estimate, assuming that the Fourier multiplier $A$ has Fourier coefficients in $\ell^{\infty}(\Z)$.
\begin{proposition}[A-priori estimate]
There exists a constant $K = K(\| a\|_{\ell^\infty}) > 0$ so that, if $u\in H^{\infty}(\T)$ is a smooth solution of \eqref{ILW 2} and $\kappa,T>0$ satisfy
\begin{equation}
\kappa \geq C_s \big( 1 + \sup_{t\in [0,T]} \|u(t)\|_{H^s_\kappa} \big)^{\frac{1}{2\eps}} ,
\label{k0 2}
\end{equation}
then
\begin{equation}
\beta_s(\kappa;u(t)) \leq e^{Kt} \beta_s(\kappa;u(0)) \quad\text{for all }t\in [0,T] .
\label{ap est}
\end{equation}
\end{proposition}
\begin{proof}
Using \eqref{dbeta} and \eqref{beta dot}, we compute
\begin{equation}
\frac{\dd}{\dd t} \beta(\kappa;u(t))
= \int_0^{2\pi} \big( |m|^2 + m + \ol{m} \big)(x;\kappa,u) Au(x)\, \dd x .
\label{beta dot 3}
\end{equation}
Note that by \eqref{F} we have
\begin{equation*}
\| Au \|_{H^{-1/2}_\kappa} \lesssim \| u\|_{H^{-1/2}_\kappa}  
\end{equation*}
uniformly for $\kappa\geq 1$.  Combining this with \eqref{m}, we find
\begin{equation*}
\bigg| \int_0^{2\pi} m(x;\kappa,u) Au(x)\, \dd x \bigg| 
\leq \| m\|_{H^{1/2}_\kappa} \| Au \|_{H^{-1/2}_\kappa} 
\lesssim \| u \|_{H^{-1/2}_\kappa}^2 ,
\end{equation*}
and similarly for the term involving $\overline{m}$.
On the other hand, from \eqref{H1/2} and \eqref{m} we see that
\begin{align*}
\bigg| \int_0^{2\pi} |m|^2(x;\kappa,u) Au(x)\, \dd x \bigg| 
&\leq \big\| |m|^2 \big\|_{H^{1/2}_\kappa} \| Au \|_{H^{-1/2}_\kappa} \\
&\lesssim_s \kappa^{-2\eps} \|m\|_{H^{s+1}_\kappa} \|m\|_{H^{1/2}_\kappa} \| Au \|_{H^{-1/2}_\kappa}  \\
&\lesssim_s \kappa^{-2\eps} \| u \|_{H^s_\kappa} \| u \|_{H^{-1/2}_\kappa}^2 .
\end{align*}
By assumption \eqref{k0 2} on $\kappa$ we find that
\begin{equation*}
\bigg| \int_0^{2\pi} |m|^2(x;\kappa,u) Au(x)\, \dd x \bigg| 
    \lesssim_s \| u \|_{H^{-1/2}_\kappa}^2 .
\end{equation*}
Altogether, we conclude
\begin{equation}
\bigg| \frac{\dd}{\dd t} \beta(\kappa;u(t)) \bigg| \lesssim_s \| u(t) \|_{H^{-1/2}_\kappa}^2 
\label{beta dot 2}
\end{equation}
uniformly for $t\in[0,T]$.

Now we will integrate in $\kappa$.  Note that
\begin{equation*}
\int_{\kappa}^\infty \frac{\vk^{2s}}{|n| + \vk} \,\dd\vk \lesssim_s \frac{1}{(|n| + \kappa)^{2|s|}}
\quad\text{uniformly for }n\in\Z ,
\end{equation*}
and so by Fubini,
\begin{equation*}
\int_{\kappa}^\infty \|f\|_{H^{-1/2}_\vk}^2\, \vk^{2s}\,\dd\vk \lesssim_s \| f\|_{H^s_\kappa} .
\end{equation*}
Combining this with \eqref{beta dot 2} and \eqref{beta s}, we find
\begin{equation*}
\bigg| \frac{\dd}{\dd t} \beta_s(\kappa;u(t)) \bigg| \lesssim_s \| u(t) \|_{H^{s}_\kappa}^2 \lesssim_s \beta_s(\kappa;u(t))
\end{equation*}
uniformly for $t\in[0,T]$.  Therefore the estimate~\eqref{ap est} now follows from Gronwall's inequality.
\end{proof}

Together with the fact that $\beta_s(\kappa;u)$ lies at $H^s$ regularity thanks to \eqref{beta s} and with the equicontinuity criterion \eqref{equicty}, a straightforward bootstrap argument yields Theorem~\ref{t:ap}.

The proof of Theorem~\ref{t:ap ILW} follows a parallel argument. We can rewrite \eqref{ILW} in the form~\eqref{ILW 2} for the operator $A$ in \eqref{F:ILW}.  In Fourier variables, the operator $A$ has the symbol
\begin{equation}
a(n)
    = - \frac{in}{\delta}+in^2 (\coth(\delta n)-\sgn(n))=-\frac{in}{\delta}+\frac{ 2 in|n|}{e^{2|\delta n|}-1}.
    \label{smoothing}
\end{equation}
Using this, it is straightforward to verify (see, for example, \cite{Chapouto2024}*{Lem.~2.1}) that for any $s\in\R$, the operator $A+\del^{-1} \partial_x$ is bounded on $H^s$ and satisfies
\begin{equation}
\| (A+\delta^{-1}\partial_x)f \|_{H^s} \lesssim_s \del^{-2} \| f\|_{H^s} \quad\text{uniformly for }\del>0.
\label{smoothing 2}
\end{equation}
On the other hand, we are able to replace $A$ by $A+\delta^{-1}\partial_x$ in \eqref{beta dot 3} since
\begin{equation*}
\dd_u \beta[ \partial_x u ] = 0 .
\end{equation*}
(Indeed, this simply expresses the fact that $\beta$ is invariant under spatial translations; see \cite{Killip2024}*{Lem.~4.1} for details.)  

Altogether, we conclude that the solutions $u^\del(t)$ of \eqref{ILW} satisfy \eqref{k0 2} with a constant $K$ that can be chosen uniformly in $\del\geq 1$.  Theorem~\ref{t:ap ILW} then follows from the same bootstrap argument as Theorem~\ref{t:ap}.

\section{Global well-posedness for $-\frac 12<s<0$}\label{s:WP}

The goal of this section is to demonstrate that \eqref{ILW 2} is globally well-posed in $H^s_0(\T)$ for $-\frac12<s<0$, and so finish the proof of Theorem~\ref{t:wp 2}.  As we will see below, this will follow easily from the following result.
\begin{theorem}
\label{t:lwp-precise}
Fix $-\frac12<s<0$.  Let $(u_j(0))_{j\geq 1} \subset H^\infty_0(\T)$ be a sequence of real-valued initial data that converges in $H^s_0(\T)$.  Then, for any $T>0$, the corresponding solutions $u_j(t)$ to \eqref{ILW 2} converge in $\mc{C}([-T,T];H^s_0(\T))$.
\end{theorem}

Let 
\[
\mc{F} = \{ u_j(0) : j\geq 1\}.
\]
Fix $T>0$, and let $\mc{F}^*_T$ be defined as in \eqref{F*}.  By Theorem~\ref{t:ap}, the set $\mc{F}^*_T$ is relatively compact in $H^s(\T)$.  As the Birkhoff map $\Phi:H^s_0(\T)\to\h^{s+\frac 12}$ is continuous, it follows that the set
\begin{equation}
\{ \Phi(u) : u\in\mc{F}^*_T \}
\label{F* 2}
\end{equation}
is relatively compact in $\h^{s+\frac12}$, and hence is bounded and tight (see also~\eqref{equicty 3}). We can therefore split the Birkhoff coordinates between low and high frequencies. We will estimate the low frequencies in the weaker norm $\mc{C}_T\h^{-1}$ via the following lemma. Then, tightness will enable us to treat the high frequency part as a small remainder term.

\begin{lemma}\label{lem:low}
For any $T_0\in(0,T]$, we have
\begin{equation}
\| \Phi(u_j(t)) - \Phi(u_k(t)) \|_{\mc{C}_{T_0}\h^{-1}} \lesssim \| u_j(0) - u_k(0) \|_{H^s} + T_0 \| u_j(t) - u_k(t) \|_{\mc{C}_{T_0}H^s} ,
\label{lwp 5}
\end{equation}
where the implicit constant is uniform in $j,k\geq 1$ and on the larger time interval $[-T,T]$.
\end{lemma}
\begin{proof}
We go back to the Duhamel formula~\eqref{duhamel}:
\begin{equation*}
\zeta_n(u(t)) = e^{i\Omega_n(t;u)} \zeta_n(u_0) +  \int_0^t e^{i[\Omega_n(t;u) - \Omega_n(\tau;u)]} \dd_{u(\tau)}\zeta_n [Au(\tau)]\, \dd\tau .
\end{equation*}
By \eqref{dzeta}, \eqref{dzeta 2}, and relative compactness of the family $\mc{F}_T^*$, we have the uniform bounds
\begin{align}
\| \Phi(u_j(t))\|_{\mc{C}_{T_0}\h^{s+\frac12}}
&\lesssim 1,
\label{unif 1} \\
\| \Phi(u_j(t)) - \Phi(u_k(t)) \|_{\mc{C}_{T_0}\h^{s+\frac12}}
&\lesssim \| u_j(t) - u_k(t) \|_{\mc{C}_{T_0}H^s} , 
\label{unif 2}\\
\| \dd_{u_j(t)}\Phi[f] \|_{\mc{C}_{T_0}\h^{s+\frac12}} 
&\lesssim \|f\|_{H^s} , 
\label{unif 3}\\
\| \dd_{u_j(t)}\Phi\, [f] - \dd_{u_k(t)}\Phi[f] \|_{\mc{C}_{T_0}\h^{s+\frac12}} 
&\lesssim \| u_j(t) - u_k(t) \|_{\mc{C}_{T_0}H^s} \| f\|_{H^s}
\label{unif 4}
\end{align}
for every $j,k\geq 1$ and $T_0\leq T$.

For $n\geq 1$, $j,k\geq 1$, and $|t|\leq T_0$, we estimate
\begin{align*}
\big| \tfrac1n [ \omega_n(u_j(t)) - &\omega_n(u_k(t)) ] \big|\\
&\lesssim \sum_{m\geq 0} \tfrac{\min\{m,n\}}{n} \Big| |\zeta_m(u_j(t))|^2 - |\zeta_m(u_k(t))|^2 \Big| \\
&\lesssim \sum_{m\geq 0} \Big| |\zeta_m(u_j(t))|^2 - |\zeta_m(u_k(t))|^2 \Big| \\
&\leq \big( \| \Phi(u_j(t)) \|_{\mc{C}_{T_0}\h^{0}} + \| \Phi(u_k(t)) \|_{\mc{C}_{T_0}\h^{0}} \big) \| \Phi(u_j(t)) - \Phi(u_k(t)) \|_{\mc{C}_{T_0}\h^{0}},
\end{align*}
so that the embedding $\h^0\hookrightarrow \h^{s+\frac 12}$ and the uniform bounds \eqref{unif 1}--\eqref{unif 2} lead to
\begin{align*}
\big| \tfrac1n [ \omega_n(u_j(t)) - \omega_n(u_k(t)) ] \big|
&\lesssim \| u_j(t) - u_k(t) \|_{\mc{C}_{T_0}H^s}.
\end{align*}
  As the function $\theta\mapsto e^{i\theta}$ is Lipschitz continuous, we have
\begin{align}
\| \tfrac{1}{n} [ e^{i\Omega_n(t;u_j)} - e^{i\Omega_n(t;u_k)} ] \|_{\mc{C}_{T_0}\ell^\infty_n}
&\lesssim \| \tfrac{1}{n} [ \Omega_n(t;u_j) - \Omega_n(t;u_k) ] \|_{\mc{C}_{T_0}\ell^\infty_n} \nonumber \\
&\lesssim T_0 \| \tfrac{1}{n} [ \omega_n(u_j(t)) - \omega_n^k(u_j(t)) ] \|_{\mc{C}_{T_0}\ell^\infty_n} \nonumber \\
&\lesssim T_0 \| u_j(t) - u_k(t) \|_{\mc{C}_{T_0}H^s} .
\label{lwp 4}
\end{align}

From \eqref{duhamel} we see that
\begin{align}
\zeta_n(&u_j(t)) - \zeta_n(u_k(t))
\nonumber \\
&= \big(e^{i\Omega_n(t;u_j)} -e^{i\Omega_n(t;u_k)}\big) \zeta_n(u_j(0))
\label{lwp 1} \\
&\phantom{{}={}}{}+ e^{i\Omega_n(t;u_k)} (\zeta_n(u_j(0))-\zeta_n(u_k(0)))
\label{lwp 1bis} \\
&\phantom{{}={}}{}+ \int_0^t \Big( e^{i(\Omega_n(t;u_j) - \Omega_n(\tau;u_j))} - e^{i(\Omega_n(t;u_k) - \Omega_n(\tau;u_k))} \Big)\dd_{u_j(\tau)}\zeta_n\, [Au_j(\tau)] \, \dd\tau 
\label{lwp 2} \\
&\phantom{{}={}}{}+ \int_0^t e^{i(\Omega_n(t;u_k) - \Omega_n(\tau;u_k))} \big( \dd_{u_j(\tau)}\zeta_n\, [Au_j(\tau)] - \dd_{u_k(\tau)}\zeta_n\, [Au_k(\tau)] \big) \, \dd\tau .
\label{lwp 3} 
\end{align}
We will estimate the four terms on the RHS individually.

For \eqref{lwp 1}, we use \eqref{lwp 4} to estimate
\begin{align*}
\| \eqref{lwp 1} \|_{\mc{C}_{T_0}\h^{-1}}
&\leq \| ( e^{i\Omega_n(t;u_j)} - e^{i\Omega_n(t;u_k)} ) \zeta_n(u_j(0)) \|_{\mc{C}_{T_0}\h^{-1}} \\
&\lesssim \big\| \tfrac1n (e^{i\Omega_n(t;u_j)} - e^{i\Omega_n(t;u_k)} ) \big\|_{\mc{C}_{T_0}\ell^\infty_n}  \| \zeta_n(u_j(0)) \|_{\mc{C}_{T_0}\h^0}  \\
&\lesssim T_0\| u_j(t) - u_k(t) \|_{\mc{C}_{T_0}H^s} .
\end{align*}
For \eqref{lwp 1bis}, we bound
\begin{align*}
\|\eqref{lwp 1bis}\|_{\mc{C}_{T_0}\h^{-1}}
&\leq \| e^{i\Omega_n(t;u_k)} (\zeta_n(u_j(0)) - \zeta_n(u_k(0)) )\|_{\mc{C}_{T_0}\h^{-1}}\\
&\leq  \| \zeta_n(u_j(0)) - \zeta_n(u_k(0)) \|_{\mc{C}_{T_0}\h^0}\\
&\lesssim \| u_j(0) - u_k(0) \|_{H^s} .
\end{align*}
For \eqref{lwp 2}, we use \eqref{lwp 4} and \eqref{unif 3} to obtain
\begin{align*}
\| \eqref{lwp 2} \|_{\mc{C}_{T_0}\h^{-1}}
& {}\lesssim T_0 \big\| \big( e^{i\Omega_n(t;u_j)} - e^{i\Omega_n(t;u_k)} \big)\dd_{u_j(t)}\zeta_n\, [Au_j(t)] \big\|_{\mc{C}_{T_0}\h^{-1}}  \\
&{}\lesssim T_0 \big\| \tfrac1n (e^{i\Omega_n(t;u_j)} - e^{i\Omega_n(t;u_k)} ) \big\|_{\mc{C}_{T_0}\ell^\infty_n}  \| \dd_{u_j(t)}\zeta_n\, [Au_j(t)] \|_{\mc{C}_{T_0}\h^{0}}  \\
&\lesssim T_0\| u_j(t) - u_k(t) \|_{\mc{C}_{T_0}H^s} .
\end{align*}
Finally, for \eqref{lwp 3} we also use \eqref{unif 4} to deduce
\begin{align*}
\| \eqref{lwp 3} \|_{\mc{C}_{T_0}\h^{-1}}
&{}\lesssim T_0 \big\| \dd_{u_j(t)}\zeta_n\, [Au_j(t)] - \dd_{u_k(t)}\zeta_n\, [Au_k(t)] \big\|_{\mc{C}_{T_0}\h^{-1}}  \\
& {}\lesssim T_0 \big\| \dd_{u_j(t)}\zeta_n\, [Au_j(t)] - \dd_{u_k(t)}\zeta_n\, [Au_j(t)] \big\|_{\mc{C}_{T_0}\h^{-1}}  \\
&\phantom{{}\lesssim} + T_0 \big\|  \dd_{u_k(t)}\zeta_n\, [Au_j(t) - Au_k(t)] \big\|_{\mc{C}_{T_0}\h^{-1}}  \\
&\lesssim T_0\| u_j(t) - u_k(t) \|_{\mc{C}_{T_0}H^s} .
\end{align*}
Adding the previous four estimates for \eqref{lwp 1}--\eqref{lwp 3} together we obtain \eqref{lwp 5}, as desired.
\end{proof}

\begin{proof}[Proof of Theorem~\ref{t:lwp-precise}]
We use the properties of the inverse map $\Phi^{-1}$ to estimate
\begin{equation*}
\| u_j(t) - u_k(t) \|_{\mc{C}_{T_0}H^s} \lesssim \| \Phi(u_j(t)) - \Phi(u_k(t)) \|_{\mc{C}_{T_0}\h^{s+\frac12}}.
\end{equation*}
For $N\geq 1$, 
it will be convenient for us to set
\begin{equation*}
\| z_n \|_{\h^r_{\geq N}} := \| \un_{n\geq N} z_n \|_{\h^r} .
\end{equation*}
With this notation in hand,
we cut between low and high frequencies:
\begin{equation*}
\| u_j(t) - u_k(t) \|_{\mc{C}_{T_0}H^s} 
    \lesssim N^{s+\frac32} \| \Phi(u_j(t)) - \Phi(u_k(t)) \|_{\mc{C}_{T_0}\h^{-1}} + 2\sup_{u\in\mc{F}^*_T}\,\| \Phi(u)\|_{\mc{C}_{T_0}\h^{s+\frac12}_{\geq N}} ,
\end{equation*}
then  we use Lemma~\ref{lem:low}:
\begin{multline*}
\| u_j(t) - u_k(t) \|_{\mc{C}_{T_0}H^s}\\ 
    \lesssim N^{s+\frac32} \big( \| u_j(0) - u_k(0) \|_{H^s} + T_0 \| u_j(t) - u_k(t) \|_{\mc{C}_{T_0}H^s} \big) + 2\sup_{u\in\mc{F}^*_T}\,\|\Phi(u)\|_{\h^{s+\frac12}_{\geq N}} .
\end{multline*}
Fix $\del>0$.  As the set \eqref{F* 2} is tight in $\h^{s+\frac12}$, we can pick $N\gg 1$ to make the second term on the RHS less than $\del>0$.  Then, we pick $T_0\ll 1$ to conclude that
\begin{equation*}
\| u_j(t) - u_k(t) \|_{\mc{C}_{T_0}H^s}
\lesssim N^{s+\frac32} \| u_j(0) - u_k(0) \|_{H^s} + \delta .
\end{equation*}
Sending $j,k\to\infty$, we conclude
\begin{equation*}
\limsup_{K\to\infty} \sup_{j,k\geq K} \| u_j(t) - u_k(t) \|_{\mc{C}_{T_0}H^s}
\lesssim \delta .
\end{equation*}
As $\delta>0$ was arbitrary, we conclude
\begin{equation*}
\| u_j(t) - u_k(t) \|_{\mc{C}_{T_0}H^s} \to 0 \quad\text{as }\min(j,k)\to\infty .
\end{equation*}

Altogether, this proves that the sequence $(u_j)_{j\geq 1}$ is a Cauchy sequence in the space $\mc{C}([-T_0,T_0],H^s(\T))$, and therefore it converges in $H^s$ uniformly for $|t|\leq T_0$.  As all of the implicit constants are uniform on the larger time interval $[-T,T]$ and the equation \eqref{ILW 2} is invariant under time translations, this ensures that the local-in-time argument may be iterated to cover $[-T,T]$.
\end{proof}

\begin{proof}[Proof of Theorem~\ref{t:wp 2}]
Recall that we already demonstrated local well-posedness for $s\geq 0$ in Section~\ref{s:s-pos}.  To finish the proof, it suffices to consider $-\frac12<s<0$.  We want to show that the solution map $S$ for equation \eqref{ILW 2} extends uniquely from $H^\infty_0$ to a jointly continuous map $S : \R\times H^s_0(\T) \to H^s_0(\T)$.

Given initial data $u(0)\in H^s_0(\T)$, we define $S(t,u(0))$ as follows.  Let $(u_j(0))_{j\geq 1}$ be a sequence of $H^\infty_0(\T)$ functions that converge to $u(0)$ in $H^s(\T)$.  Applying the previous theorem to this sequence, we see that the corresponding solutions $u_j(t)$ to \eqref{ILW 2} converge in $H^s(\T)$ and the limit is independent of the sequence $(u_j(0))_{j\geq 1}$.  Consequently, the solution map
\begin{equation*}
S(t,u(0)) := \lim_{j\to\infty} u_j(t)
\end{equation*}
is well-defined.

It remains to show that $S$ is jointly continuous.  Fix $T>0$ and let $(u_j(0))_{j\geq 1}$ be a sequence of initial data in $H^s_0(\T)$ that converges to $u(0)$ in $H^s(\T)$.  By definition of $S$, we may choose another sequence $\wt{u}_j(t)$ of smooth solutions to \eqref{ILW 2} such that
\begin{equation*}
\| S(t,u_j(0)) - \wt{u}_j(t)\|_{\mc{C}_TH^s} \to 0 \quad\text{as }j\to\infty .
\end{equation*}
In particular, $\wt{u}_j(0) \to u(0)$ in $H^s(\T)$, and so by the previous theorem we have
\begin{equation*}
\| \wt{u}_j(t) - S(t,u(0)) \|_{\mc{C}_TH^s} \to 0 \quad\text{as }j\to\infty .
\end{equation*}
Given a sequence $(t_j)_{j\geq 1} \subset [-T,T]$ that converges to some $t\in[-T,T]$, we estimate
\begin{align*}
&\| S(t_j,u_j(0)) - S(t,u(0)) \|_{H^s}\\
&\quad \leq \| S(t_j,u_j(0)) - \wt{u}_j(t_j) \|_{\mc{C}_TH^s} + \| \wt{u}_j(t_j) - \wt{u}_j(t) \|_{H^s} + \| \wt{u}_j(t) - S(t,u(0)) \|_{\mc{C}_TH^s} .
\end{align*}
The RHS converges to zero as $j\to\infty$, and so we conclude that $S$ is jointly continuous.
\end{proof}

\section{The ILW equation}\label{s:ILW}

The goal of this section is to prove Theorem~\ref{t:conv}.  Fix $-\frac12 < s<0$ and initial data $u(0) \in H^s_0(\T)$, and consider the global solutions $u^\del(t)$ of \eqref{ILW} and $u^\infty(t)$ of \eqref{BO}.  We want to show that $u^\del$ converges to $u^\infty$ in $\mc{C}([-T,T];H^s(\T))$ for any $T>0$.

Recall that \eqref{ILW} can be put the form~\eqref{ILW 2} for the operator $A$ appearing in \eqref{F:ILW}.  Moreover, this operator $A$ satisfies the estimate \eqref{smoothing 2}.

First, we will prove that the sequence $(\Phi(u^\del))_{\del\geq 1}$ is tight in $\h^{s+\frac12}$:
\begin{lemma}[Tightness]
For any $T>0$, we have
\begin{equation}
 \sup_{1\leq \del\leq \infty}\, \| \Phi(u^\del(t) )\|_{\mc{C}_T\h^{s+\frac12}_{\geq N}}  \to 0 \quad\text{as }N\to\infty.
\label{tight}
\end{equation}
\end{lemma}
\begin{proof}
The statement for $\del=\infty$ is immediate, since by \eqref{zeta dot 3} we have
\begin{equation*}
|\zeta_n(u^\infty(t))|^2 = |\zeta_n(u(0))|^2 \quad\text{for all }t\in\R
\end{equation*}
and the initial data $\Phi(u(0))$ is a fixed sequence in $\h^{s+\frac12}$.

By Theorem \ref{t:ap ILW}, we know that the family of solutions $u^\del(t)$ for $|t|\leq T$ and $\del\geq 1$ is relatively compact in $H^s(\T)$.  Therefore, by continuity of the Birkhoff map $\Phi$, we know that the sequences $\Phi(u^{\del}(t))$ for $|t|\leq T$ and $\del\geq 1$ are relatively compact in $\h^{s+\frac12}$, and thus are tight.
\end{proof}

We recall that in~\eqref{boost}, we made the change of functions $v^\del(t,x)=u^\del(t,x+\delta^{-1}t)$ in order to fit the framework of Definition~\ref{def:A}.  In view of the following lemma, we also set $v^\infty = u^\infty$.
\begin{lemma}
If the sequence $(v^\del)_\del$ converges to $u^\infty$ in $\mc{C}([-T,T],H^s(\T))$ as $\del\to+\infty$, then so does the sequence $(u^\del)_\del$.
\end{lemma}
\begin{proof}
Thanks to Theorem~\ref{t:ap ILW}, we know that the sequence $(u^\del)_{\del\geq 1}$ is equicontinuous in $H^s(\T)$. 
For $N\geq 1$, we estimate
\begin{equation*}
\| u^\del - v^\del \|_{H^s} \leq \sup_{|n|\leq N} | \wh{u^\del}(n) - \wh{v^\del}(n) | + 2\,\sup_{\del\geq 1}\| u^\del \|_{H_{\geq N}^s} .
\end{equation*}
By equicontinuity, we may pick $N\gg 1$ to make the second term on the RHS above arbitrarily small.
It is therefore enough to show that for each $n\in \Z$, the Fourier coefficient with index $n$ is convergent: $\wh{u^{\del}}(n) - \wh{v^\del}(n) \to 0$ as $\delta\to+\infty$ uniformly for $|t|\leq T$. This latter property is true by definition of $v^\del$.
\end{proof}

Note that the sequence $(v^\del)_\del$ also satisfies the equicontinuity property. Therefore, one can split the study between low and high Birkhoff frequencies.

For the low-frequency contribution, it will suffice to prove that the sequence $(\Phi(v^\delta))_{\delta}$ converges at the lower regularity~$\h^{-1}$:
\begin{lemma}
For any $T>0$, we have
\begin{equation}
\| \Phi(v^\delta(t)) - \Phi(u^\infty(t)) \|_{\mc{C}_T\h^{-1}}\to 0 \quad\text{as }\del\to\infty.
\label{low reg}
\end{equation}
\end{lemma}

\begin{proof}
By Theorems \ref{t:ap} and \ref{t:ap ILW}, we know that the functions $\Phi(v^\delta(t))$, $\Phi(u^\infty(t))$ and their differentials are bounded in $\h^{s+\frac12}$, uniformly for $|t|\leq T$ and $\delta\geq 1$. In the following, we will allow the implicit constants to depend on this upper bound, and on the regularity exponent $s$.

Using \eqref{duhamel}, we write for $t\in[-T,T]$ and $n\geq 1$,
\begin{align}
&\zeta_n(v^\del(t)) - \zeta_n(u^\infty(t)) 
\nonumber \\
&\qquad= ( e^{i\Omega_n(t;v^\del)} - e^{i\Omega_n(t;u^\infty)} ) \zeta_n(u(0)) 
\label{low reg 1} \\
&\qquad\phantom{{}={}}{}+ \int_0^t e^{i(\Omega_n(t;v^\del) - \Omega_n(\tau;v^\del))} \dd_{v^\del(\tau)}\zeta_n [(A+\delta^{-1}\partial_x)v^\del(\tau)]\, \dd\tau.
\label{low reg 2} 
\end{align}
We will estimate the two terms on the RHS individually.

Let us start with the term \eqref{low reg 2}.  Using \eqref{dzeta} and \eqref{smoothing 2}, we estimate
\begin{equation*}
\| \eqref{low reg 2} \|_{\mc{C}_T\h^{-1}}
\leq \| \eqref{low reg 2} \|_{\mc{C}_T\h^{s+\frac12}}
\leq CT \| (A+\delta^{-1}\partial_x)v^\delta \|_{\mc{C}_T H^s} \lesssim T \del^{-2} .
\end{equation*}
The upper bound above converges to zero as $\del\to\infty$.

Next, we turn to \eqref{low reg 1}.  As the function $\theta\mapsto e^{i\theta}$ is Lipschitz continuous, we have
\begin{align*}
\| \eqref{low reg 1} \|_{\mc{C}_T\h^{-1}}
&\lesssim \| \tfrac{1}{n} ( \Omega_n(t;v^\del) - \Omega_n(t;u^\infty) ) \|_{\mc{C}_T\ell^\infty_n} \| \zeta_n(0) \|_{\h^0} \\
&\lesssim T \| \tfrac{1}{n} ( \omega_n(v^\del(t)) - \omega_n(u^\infty(t)) ) \|_{\mc{C}_T\ell^\infty_n} .
\end{align*}
By definition of $\omega_n$,
\begin{align*}
\big| \tfrac{1}{n} ( \omega_n(v^\del(t)) - \omega_n(u^\infty(t)) )   \big|
&\leq \sum_{k\geq 0} \tfrac{\min\{k,n\}}{n} \Big| |\zeta_k(v^\del(t))|^2 - |\zeta_k(u(0))|^2 \Big| \\
&\leq \sum_{k\geq 0} \Big| |\zeta_k(v^\del(t))|^2 - |\zeta_k(u(0))|^2 \Big| .
\end{align*}
However, from \eqref{zeta dot} we see that
\begin{equation}
\frac{\dd}{\dd t} |\zeta_k(v^\del(t))|^2
= 2 \Re\big\{ \ol{\zeta_k(v^\del)}\, \dd_{v^\del}\zeta_k [(A+\delta^{-1}\partial_x)v^\del] \big\} ,
\label{zeta dot 2} 
\end{equation}
and so
\begin{align*}
\sum_{k\geq 0} \Big| |\zeta_k(v^\del(t))|^2 - |\zeta_k(u(0))|^2 \Big|
&\lesssim T \bigg\| \sum_{k\geq 0} | \zeta_k(v^\del)| \, \big| \dd_{v^\del}\zeta_k[(A+\delta^{-1}\partial_x)v^\del] \big| \bigg\|_{\mc{C}_T} \\
&\lesssim T \| \Phi(v^\del) \|_{\mc{C}_T\h^0} \| \dd_{v^\del}\Phi[(A+\delta^{-1}\partial_x)v^\del]\|_{\mc{C}_T\h^0} \\
&\lesssim T \| \Phi(v^\del) \|_{\mc{C}_T\h^{s+\frac12}}  \| (A+\delta^{-1}\partial_x)v^\del \|_{\mc{C}_T H^s}.
\end{align*}
Thanks to the uniform bounds on $\Phi(v^\del)$ and~\eqref{smoothing 2}, we obtain
\begin{equation*}
\sum_{k\geq 0} \Big| |\zeta_k(v^\del(t))|^2 - |\zeta_k(u(0))|^2 \Big|
\lesssim T \del^{-2} .
\end{equation*}

Altogether, we conclude
\begin{equation*}
\| \eqref{low reg 1} \|_{\mc{C}_T\h^{-1}} \lesssim T^2\del^{-2} .
\end{equation*}
The RHS converges to zero as $\del\to\infty$, and so completes the proof.
\end{proof}

We are now prepared to upgrade this convergence to the desired regularity and prove our main result:
\begin{proof}[Proof of Theorem~\ref{t:conv}]
Fix $T>0$.  As the restriction of $\Phi$ to the set $\{ v^\del(t) : t\in[-T,T]$, $\del\geq 1\}$ is bi-Lipschitz continuous from $H^s_0(\T)$ into $\h^{s+\frac12}$, 
it suffices to show
\begin{equation}
\| \Phi(v^\del(t)) - \Phi(u^\infty(t)) \|_{\mc{C}_T\h^{s+\frac12}} \to 0 \quad\text{as }\del\to\infty .
\label{conv}
\end{equation}

For $N\geq 1$, we bound
\begin{multline*}
\| \Phi(v^\del) - \Phi(u^\infty ) \|_{\mc{C}_T\h^{s+\frac12}}
\lesssim N^{s+\frac32} \| \Phi(v^\del) - \Phi(u^\infty) \|_{\mc{C}_T\h^{-1}} + \sup_{1\leq\del\leq\infty} \| \Phi(v^\del )\|_{\mc{C}_T\h^{s+\frac12}_{\geq N}} .
\end{multline*}
By \eqref{tight}, we may choose $N\gg 1$ to make the second term on the RHS arbitrarily small.  Once $N$ is fixed, we then know that the first term converges to zero as $\del\to\infty$ by \eqref{low reg}.  This proves \eqref{conv}, and so completes the proof of Theorem~\ref{t:conv}.
\end{proof}

\bibliography{refs}
\end{document}